\documentclass[11pt]{amsart}

\usepackage{amssymb, amsmath}
\usepackage{diagbox}

\usepackage{tikz, tikz-cd}
\usepackage{enumitem}
\usepackage{mathtools}
\usepackage{todonotes}
\usepackage{amsthm}
\usepackage{amsfonts}
\usepackage{amsmath}
\usepackage[all]{xy}

\theoremstyle{plain}
 \newtheorem{thm}{Theorem}[section]

\newtheorem{prop}[thm]{Proposition}

\theoremstyle{definition}

\theoremstyle{plain}
\newtheorem{theorem}[thm]{Theorem}

\theoremstyle{definition}

\newtheorem{defin}[thm]{Definition}

\numberwithin{equation}{section}

\newcommand{\sH}{{\mathcal H}}

\newcommand{\sL}{{\mathcal L}}

\newcommand{\B}{{\mathbb B}}
\newcommand{\C}{{\mathbb C}}

\newcommand{\E}{{\mathbb E}}

\newcommand{\N}{{\mathbb N}}

\newcommand{\R}{{\mathbb R}}

\newcommand{\hol}{\ensuremath{\mathcal{O}}}

\newcommand\s{\sigma}

\newcommand\e{\epsilon}

\DeclareMathOperator{\Pic}{Pic}

\newcommand{\HH}{\ensuremath{\mathbb{H}}}

\newcommand{\FF}{\ensuremath{\mathbb{F}}}

\newcommand{\ra}{\ensuremath{\rightarrow}}

\newcommand{\CC}{\mathbb{C}}

\newcommand{\PP}{\mathbb{P}}
\newcommand{\QQ}{\mathbb{Q}}
\newcommand{\RR}{\mathbb{R}}
\newcommand{\ZZ}{\mathbb{Z}}

\newcommand{\Num}{\mathrm{Num}}

\begin{document}
\title[
Odd fake $\QQ$-homology quadrics]{
Odd fake  $\QQ$-homology quadrics exist}
\author{Fabrizio Catanese}
\date{\today}

\address{Fabrizio Catanese,
 Mathematisches Institut der Universit\"{a}t
Bayreuth, NW II\\ Universit\"{a}tsstr. 30,
95447 Bayreuth, Germany.}
\email{Fabrizio.Catanese@uni-bayreuth.de}


\maketitle

\begin{abstract}
We show the existence of odd fake $\QQ$-homology quadrics, namely of minimal surfaces $S$ of general type
which have the same $\QQ$-homology as a smooth quadric $Q \cong \PP^1(\CC)^2$,
but have an odd intersection form on $ H^2(S, \ZZ)/\mathrm{Tors}(S)$, where $\mathrm{Tors}(S)$ is the Torsion subgroup.

Our examples are provided  by a special  1-dimensional family of  surfaces isogenous to a product of unmixed  type.

\end{abstract}

\section{Introduction}

Minimal smooth complex algebraic surfaces  of general type which have the same $\QQ$-homology algebra as a quadric $Q \cong \PP^1(\CC)^2$
are called fake  $\QQ$-homology quadrics. They are  exactly  the  smooth minimal complex surfaces with  $q=p_g=0$,  with  second Betti number $b_2(S)=2$
(equivalently, with $K^2_S=8$) and with nef canonical divisor.

Their  Severi lattice $Num(S)$, of divisor classes modulo numerical equivalence,  has rank 2 and coincides   with $ H^2(S, \ZZ)/\mathrm{Tors}(S)$;
 by Poincar\'e duality, it has a unimodular intersection form. The intersection  form can be even (then the lattice is called  a hyperbolic plane), or it is odd and then  the intersection form is diagonalizable (see \cite{serre}).

In the paper \cite{ccck}, among other things, many examples were produced of such fake  $\QQ$-homology quadrics with  even intersection form,
but we were not able to show the existence of such surfaces with 
odd intersection form.

This article provides such  a family of  such surfaces, 
thus answering problems 1 and 3 raised in \cite{ccck}.

The  family consists of  surfaces $S$ isogenous to a higher product of unmixed type,
which  means that  
$$S\cong (C_1 \times C_2)/G,$$
where $C_i$ is a smooth curve with $g(C_i)\ge 2$ and $G$ is a finite group acting faithfully on $C_1$ and on $C_2$,
 such that  the diagonal action $ g(x,y) : = (gx, gy)$  of $G$ on $C_1\times C_2$ is free.

In particular, the universal covering of $S$  the bidisk $\HH \times\HH$.

Indeed, up to now, all the known  examples of fake homology quadrics have universal cover equal to the bidisk $\HH \times\HH$,
and one may even conjecture that it must be always so.

The more narrow class of topologically fake quadrics  refers to an old question raised by  Hirzebruch (\cite{hirz} Problem 25,  see also   \cite{ges-werke}, pages 779-780): does there  exist a surface of general type which is homeomorphic to $\PP^1 \times \PP^1$?  \footnote{ One can ask the same question for a surface
  homeomorphic to the blow up $\FF_1$ of $\PP^2$ in one point: in this case the intersection form is odd.} 

Here  is the precise form of our main result.

\begin{theorem}\label{main-theorem}
Consider the  family of surfaces isogenous to a product of unmixed type with $q=p_g=0$, 
 with group $G := \frak A_5$, and where 

$C_1$ is the $G$-covering of $\PP^1$ branched on  three points with Hurwitz triple (monodromy triple)
$(2,4)(3,5) ,  (2,1,3,4,5),  (1,2,3,4,5), $

$C_2$ is the $G$-covering of $\PP^1$ branched on four points with Hurwitz quadruple (monodromy quadruple)
$(1,2,3), (3,4, 5) ,  (4,3,2),  (2,1,5). $

Then the corresponding family of surfaces  isogenous to a product of unmixed type,
$S = (C_1 \times C_2)/G$,  is a family of fake $\QQ$-homology quadrics of odd type.

\end{theorem}

The proof hinges on the symmetries of the curve $C_1$, the unique curve of genus $4$ with $\frak A_5$-symmetry.

An interesting question would be to determine exactly which other of the known fake $\QQ$-homology quadrics
are of odd type (see the examples of \cite{product},  \cite{ccck}, \cite{frapporti}, \cite{dzambic}, \cite{d-r}).

   {\bf Acknowledgement.} 
   
   Heartfelt  thanks to Matthias Sch\"utt for making the crucial observation that the surfaces in the first family of  the classification
   list  of surfaces isogenous to a product with $q=p_g=0$ have Torsion group of odd order. 
   
   The question about the  existence of fake quadrics of odd type was raised by Jianqiang Yang in an email correspondence 
   (in December 2022). 
 
 Thanks to Kyoung-Seog Lee for asking a question concerning the proof, which lead me to 
 add  a few lines  of explanation before Proposition 3.1.

\section{Definitions and basic properties}

We shall  work with projective smooth surfaces defined over the field $\CC$.

\begin{defin}[ $\QQ$-homology Quadric]\label{defn:rationalquadric}
Let $S$ be a smooth projective surface over $\CC$.

The surface $S$ is called a   {\bf $\QQ$-homology quadric} if $H^*(S, \QQ)$ has Betti numbers $b_1(S)=b_3(S)= 0$, 
 $b_2(S)=2$ (equivalently, $q(S)=p_g(S)=0$).

In turn, it will be called an {\bf even homology  quadric} if moreover
 the intersection form on $\mathrm{Num}(S)$ is even, that is, with matrix
    \begin{align*}
        \left(\begin{array}{cc}
               0 & 1\\
               1 & 0
              \end{array}\right).
    \end{align*}

It will be called an {\bf odd homology quadric } if instead the intersection form is odd (hence diagonalizable with diagonal entries $ (+1, -1)$).
\end{defin}
For a $\QQ$-homology quadric $S$, by  the Noether formula we must have
$$\chi(\hol_S)=1, c_2(S)=4, K_S^2=8.$$
Also the long exact sequence of cohomology groups associated to the exponential exact sequence shows that, if $b_1(S)=0$,
 $$c_1 \colon \mathrm{Pic}(S) \rightarrow H^2(S,\ZZ)$$
 is an isomorphism. We identify these two groups and denote by
 $\mathrm{Tors}(S)$ their torsion subgroup. Then it follows that 
 $$\mathrm{Num}(S)\cong \mathrm{Pic}(S)/\mathrm{Tors}(S) \cong H^2(S,\ZZ)/\mathrm{Tors}(S).$$

An even homology   quadric  $S$ must be minimal, since it contains no $(-1)$-curves. Hence, by surface classification, for an even homology quadric
\begin{itemize}
    \item either $S$ is rational, and then $S\cong \mathbb{F}_{2n}$ for $n\ge 0$;
    \item or $S$ is a minimal surface of general type.
\end{itemize}
In the former case $S$ is simply connected, $\mathrm{Tors}(S)=0$ and $K_S$ is not ample.

 In the case of an odd homology quadric
\begin{itemize}
    \item either $S$ is rational, and then $S\cong \mathbb{F}_{2n+1}$ for $n\ge 0$;
    \item or $S$ is a, non necessarily minimal, surface of general type.
\end{itemize}

\begin{defin}[Even fake Quadric]\label{defn:fakequadric} Let $S$ be a smooth projective surface over $\CC$.
The surface $S$ will be  called \footnote{The reader should be made aware that, in the literature, sometimes $\QQ$-homology quadrics
of general type  are referred to as fake quadrics, without further specification whether the intersection form is even or odd, see \cite{ges-werke}, \cite{dzambic}, \cite{d-r}, \cite{f-l}.} an {\bf even fake quadric} if it is an even  homology quadric and is of general type.
\end{defin}

\begin{defin}[Odd fake Quadric]\label{defn:fakeF1}Let $S$ be a smooth projective surface over $\CC$.
  If $S$  is an odd homology quadric which is  of general type,
then either $S$ is not minimal and it is a one point blow up of a fake projective plane, or
we call the  surface $S$ an {\bf odd fake quadric}, which means:
\begin{enumerate}
    \item $S$ is minimal of general type with $K_S^2=8, p_g(S)=0$;
    \item the intersection form on $\mathrm{Num}(S)$ is odd.
\end{enumerate}
\end{defin}

Therefore smooth minimal surfaces of general type with $K^2=8$ and $p_g=0$ are divided into two classes:
the even and the odd fake quadrics.

\begin{defin}Let $S$ be a smooth projective surface.
The surface $S$ is said to be {\bf isogenous to a  higher product}  \cite{isogenous}  if
$$S\cong (C_1 \times C_2)/G,$$
where $C_i$ is a smooth curve with $g(C_i)\ge 2$ and $G$ is a finite group acting faithfully and freely on
$C_1\times C_2$.

 If there are respective actions of $G$ on $C_1$ and $C_2$ such that  $G$ acts
 by the diagonal action $ g(x,y) : = (gx, gy)$  on $C_1\times C_2$, then $S$ is called of {\bf unmixed type}.

If some element of $G$ exchanges two factors, then $S$ is called of {\bf mixed type}.
\end{defin}
 Surfaces isogenous to a product of unmixed type with $p_g=0$,
 have been classified
in \cite{product} (their torsion groups have been classified in \cite{bcf}).

We summarize here, for the reader's benefit, their classification, their torsion groups $Tors(S) \cong H_1(S, \ZZ)$, and an extra bit of information, namely the parity
of the intersection form when it is  known: the assertions `even'  were proven in \cite{ccck}, while the assertion `odd' will be the object of our main
theorem.

\begin{theorem}\label{list}
Let $S = (C_1\times C_2)/G$ be a surface isogenous to a product of unmixed type,
with $p_g(S) = 0$; then $G$ is one of the groups in the Table \ref{tab1} and the multiplicities $T_1, T_2$ of the multiple fibres for
 the two natural fibrations $ S \ra C_i/G$ are listed in the table.
 For each case in the list we have an irreducible component of the moduli space of surfaces of general type, whose dimension is denoted by
 $D$. The property of being an even, respectively an odd  homology quadric
   and the first homology group of $S$ are given in the third last  column, respectively in the last column. 
  \footnote{ The group identity $Id(G)$ consists of the group cardinality $|G|$ followed by the Atlas list number.}

\begin{table}[!h]
\begin{tabular}{c|c|c|c|c|c|c}
$G$ & $Id(G)$ & $T_1$ & $T_2$ & $parity$&$D$&  $H_1(S,\mathbb Z)$\\
\hline
$\frak A_5$ & $\langle 60,5\rangle$& $[2,5,5]$& $[3,3,3,3]$&odd&1&$(\mathbb Z_3)^2\times (\mathbb Z_{15})$\\
$\frak A_5$ & $\langle 60,5\rangle$& $[5,5,5]$& $[2,2,2,3]$  &?&1&$(\mathbb Z_{10})^2$\\
$\frak A_5$ & $\langle 60,5\rangle$& $[3,3,5]$& $[2,2,2,2,2]$  &? &2&$(\mathbb Z_2)^3\times \mathbb Z_6$\\
$\frak S_ 4 \times \mathbb Z_2$& $\langle 48,48 \rangle$& $[2,4,6]$& $[2,2,2,2,2,2]$  &?&3&
$(\mathbb Z_2)^4\times \mathbb Z_4$\\
G(32) & $\langle 32,27 \rangle$& $[2,2,4,4]$& $[2,2,2,4]$ &?&2&$(\mathbb Z_2)^2\times \mathbb Z_4\times \mathbb Z_8$\\
$(\mathbb Z_5)^2$ & $\langle 25,2\rangle$& $[5,5,5]$& $[5,5,5]$ &even&0&$(\mathbb Z_5)^3$\\
$\frak S_ 4 $ & $\langle 24,12\rangle$& $[3,4,4]$& $[2,2,2,2,2,2]$ &even&3&$(\mathbb Z_2)^4\times \mathbb Z_8$\\
G(16) & $\langle 16,3\rangle$& $[2,2,4,4]$& $[2,2,4,4]$&even&2&$(\mathbb Z_2)^2\times \mathbb Z_4\times \mathbb Z_8$\\
$D_4\times \mathbb Z_2$ & $\langle 16,11\rangle$& $[2,2,2,4]$& $[2,2,2,2,2,2]$ &?&4&$(\mathbb Z_2)^3\times (\mathbb Z_4)^2$\\
$(\mathbb Z_2)^4 $ & $\langle 16,14\rangle$& $[2,2,2,2,2]$&$[2,2,2,2,2]$  &even &4&$(\mathbb Z_4)^4$\\
$(\mathbb Z_3)^2$ & $\langle 9,2\rangle$& $[3,3,3,3]$& $[3,3,3,3]$&even&2&$(\mathbb Z_3)^5$\\
$(\mathbb Z_2)^3$ & $\langle 8,5\rangle$& $[2,2,2,2,2]$& $[2,2,2,2,2,2]$  &?&5&$(\mathbb Z_2)^4\times (\mathbb Z_4)^2$\\
\end{tabular}	
\caption{ }\label{tab1}
\end{table}
\end{theorem}

\section{Proof of the main Theorem \ref{main-theorem}}

\begin{proof}

In view of the cited results in \cite{product} and \cite{bcf}, it suffices to prove the assertion about the intersection form
being odd.

As observed in Lemma 2.6 of \cite{ccck}, since for each divisor $D$ $$D^2 + D \cdot K_S = 2 p(D) -2,$$
and $H^2(S, \ZZ)/\mathrm{Tors}(S) = Num (S)$, the intersection form is even if and only if $D \cdot K_S$ is even, for all $D$.
By unimodularity of the intersection form, the intersection form is even if and only if the class of $K_S$
in $Num(S)$ is divisible by 2.

In this special case,  since the Torsion subgroup  (see the table in Theorem \ref{list}) is of odd order, then the 
the class of $K_S$  is divisible by 2 if and only if $K_S$ is divisible by 2 (actually the same  holds for each divisor class).

Using the ramification formula for $ S \ra (C_1/G) \times (C_2/G) = \PP^1  \times  \PP^1 $,
we see that, setting $m_1, \dots, m_s$ to be the multiplicities of the multiple fibres in the first fibration,
and $ n_1, \dots, n_{s'}$ those in the second, then
$$ K_S = (-2 + \sum_j (1 - \frac{1}{m_j} )) F_1 +  (-2 + \sum_i (1 - \frac{1}{n_i} ) ) F_2 .$$

We  use now (3) of Lemma 2.4 of \cite{ccck} showing that in general $F_j = d_j \Phi_j$, where $d_j$ is the least common multiple of the multiplicities of the fibration $S \ra C_j /G$, whose  general fibre is denoted $F_j$.

Hence, in $Num(S) $, we have, for our special family, 
$$ K_S  \equiv  \Phi_1 + 2 \Phi_2 .$$

Then $[K_S] \in Num(S)$ is not 2-divisible if and only if  $\Phi_1$ is not 2-divisible.

In turn, this is equivalent to the fact that the canonical divisor $K_{C_1}$, which is the pull back of
$\Phi_1$,  is not linearly equivalent to  $2 \theta$, where $\theta$ is a  $G$-linearized line bundle on $C_1$.

In fact,  $\Phi_1$ is 2-divisible in $Pic(S)$ if and only if  on $C_1 \times C_2$ there is a 
$G$-linearized square root $\sL$
of the line bundle  $p_1^* (K_{C_1})$. 

 $\sL$ would be of the form  $p_1^* (\theta)  + p_2^* (M)$,
where $M$ is a 2-torsion line bundle on $C_2$. Since $M$ is a torsion line bundle, it is a pull back
form $Jac(C_2)$, and as such it admits a $G$-linearization for the induced action of $G$ on $Jac(C_2)$.
Hence the existence of a $G$-linearization on $\sL$ is equivalent to the existence
of a  $G$-linearization on $\theta$.

The proof is concluded via   showing the following Proposition.

\end{proof}

\begin{prop}
The curve $C_1$ does not admit any $G := \frak A_5$-linearized Thetacharacteristic (that is, a divisor $\theta$
such $ 2 \theta \equiv K_{C_1}$).

$C_1$ is the unique genus 4 curve with automorphism group containing $\frak A_5$, canonically embedded as:
$$ C_1 = \{ (x_1, \dots, x_5) \in \PP^4 | \sum_i x_i=0 , \sum_i x_i^2=0, \sum_i x_i^3=0  \} .$$

\end{prop}

\begin{proof}
For simplicity of notation, denote here $C_1$ by $C$,  $K : = K_{C_1}$, $\frak A_5$ by $G$.

$C$ is a curve of genus $g=4$, since the canonical divisor has degree $ 6 = \frac{60}{10}$.

Observe that any $G$- invariant holomorphic 1-form descends to $ C/ G = \PP^1$, hence $ H^0(K)^G = 0$.

STEP I) Since  the irreducible nontrivial  linear representations of $G$ have respective dimensions $3,3,4,5$ (see \cite{liebeck}, page 221),
it follows that $ H^0(K)$ is the unique irreducible representation of dimension 4, the permutation representation 
$$V :=  \{ (x_1, \dots, x_5) \in \C^5 | \sum_i x_i=0\}.$$

Let $\theta$ be a Thetacharacteristic on $C$. Then  $\theta$ has degree $3$ and $h^0(\theta) = h^1(\theta) \leq 2$.

ASSUMPTION : the divisor class of $\theta$ is $G$-invariant, hence we have (see \cite{mumford})
the Theta-group $$ 1 \ra \CC^* \ra Theta(\theta) \ra G \ra 1,$$
the group of automorphisms of the line bundle $\theta$ which are lifts of automorphisms of $G$.
The above exact sequence  splits if and only if $\theta$ admits a $G$-linearization, in which case $H^0(\theta)$
is a linear $G$-representation.

STEP II) $C$ is not hyperelliptic, since otherwise, if $\sH$ is the hyperelliptic line bundle, $S^2 (H^0 (\sH)) \subset H^0(K)$
would be a $G$- invariant subspace of dimension 3,  contradicting  the  irreducibility of $H^0(K)$.

STEP III) If $\theta$ admits a $G$-linearization, then $ dim (H^0(\theta))=0$, since  otherwise again 
$S^2 (H^0 (\theta)) \subset H^0(K)$
would be a $G$- invariant subspace of dimension $\leq 3$.

STEP IV) If $C$ has genus 4 and $\frak A_5$-symmetry, then 
$$ C = \{ (x_1, \dots, x_5) \in \PP^4 | \sum_i x_i=0 , \sum_i x_i^2=0, \sum_i x_i^3=0  \} ,$$
and $ Aut(C) = \frak S_5$.

{\em Proof of Step IV)}.

The Hurwitz formula easily shows that $C / \frak A_5 \cong \PP^1$, since  the maximal order of an element of $G$ is 5.

Since $C$ is non hyperelliptic, $C$ is the complete intersection of a quadric $Q$ and of a cubic $M$
inside $\PP^3 = \{ (x_1, \dots, x_5) \in \PP^4 | \sum_i x_i=0 \}. $

Since $Q$ is unique, up to scalars, and also $M$ is unique modulo $Q$, while $G$ has no other irreducible
representations of dimension 1 than the trivial one, $Q$ and $M$ must be symmetric polynomials.

Indeed, letting  
$\chi_2 : =  \chi \circ(g \mapsto g^2)$, $S^2(H^0(K)) = S^2(V)$ decomposes, having character $\frac{1}{2} (\chi^2 + \chi_2)$
which takes values $ 
(10,1,2,0,0)$ on the conjugacy classes, as the sum $\CC \oplus V \oplus U$, where $U$ is the irreducible 5-dimensional representation.

This reproves that there is only one invariant quadric, which may be chosen to be  $\sum_i x_i^2=0$,
by the Newton formulae relating Newton and elementary symmetric polynomials. The same argument shows that
$M$ can be chosen as $\sum_i x_i^3=0$. In particular $C$ admits an $\frak S_5$-symmetry, hence 
$ Aut (C) = \frak S_5$, because of the Hurwitz-Klein  bound for $g=4$, $Aut(C) \leq 168$.

\qed

STEP V) Geometry of the quadric containing the canonical curve and projective representations of $G$.

Here $ C \subset Q = \PP^1 \times  \PP^1 $, and the last product decomposition amounts to seeing 
$K = L_1 \otimes L_2$, where the two line bundles $L_j$ are corresponding to the two coordinate projections
of $\PP^1 \times  \PP^1 $, that  is, $H^0(L_1)$, respectively $H^0(L_2)$ yield the two morphisms $C \ra \PP^1$.

Since $G$ has no subgroup of index 2, then each line bundle class $L_j$ is $G$-invariant,
and $H^0(L_j)$ is a projective representation of $G$. 

Since the group of Schur multipliers $H^2 (G, \CC ^* )= H^3 (G, \ZZ )= H_2 (G, \ZZ )= \ZZ /2$,
there is a unique central extension of $G$, equal to 
$$ 1 \ra \ZZ /2 \ra SL(2,5) \ra \PP SL(2,5) = \frak A _5 \ra 1,$$
and the projective representations of $G$ correspond to the linear representations of $SL(2,5)$.

The irreducible nontrivial  linear representations of $SL(2,5)$ have respective dimensions $2,2,3,3,4,4, 5,6$
(see \cite{dornhoff}, Section 38)
and  the centre acts non trivially on the two irreducible representations of dimension 2.

This is why  $H^0(L_1)\otimes H^0(L_2) =V$ is indeed  representation of $G$, and, since it is irreducible,
$H^0(L_1), H^0(L_2)$ are non isomorphic representations of $SL(2,5)$.

The above  discussion shows in particular  that {\bf the $G$-invariant divisor classes $L_1, L_2$ do not admit a linearization}.

STEP VI) Let $\theta$ be a $G$-invariant Thetacharacteristic: then $ \theta\otimes L_j$ is also $G$-invariant, for $j=1,2$.

Now, $ \theta\otimes L_j$ has degree $6$ but is not isomorphic to $K$, hence $H^0 ( \theta\otimes L_j)$
has dimension equal to 3.

STEP VII) $H^0 ( \theta\otimes L_j)$ is a 3-dimensional projective representation of $G$,
hence a 3-dimensional linear  representation of $SL(2,5)$; but, by the character table, the two irreducible linear representations
of $G$ of dimension 3  lift to the two irreducible linear representations
of $SL(2,5)$ of dimension 3. 

Hence either $H^0 ( \theta\otimes L_j)$ is a 3-dimensional irreducible linear  representation of $G$, or it contains a nonzero 
invariant section $\s_j$.

STEP VIII) There cannot be  two such invariant sections $\s_1, \s_2$, because their product would yield 
an invariant quadric $ q \in H^0(2K)$. We have in fact  shown that there is only one invariant quadric, and this is $Q$:
we reach a contradiction since $ Q |C \equiv 0,$  while $q$ vanishes only on a finite set of points of $C$.

STEP IX) Hence there is a  $j$ such that $H^0 ( \theta\otimes L_j)$ is a 3-dimensional irreducible linear  representation of $G$,
and then we claim that $\theta\otimes L_j$ admits a $G$-linearization.

In fact, otherwise the associated Thetagroup equals $SL(2,5)$. $SL(2,5)$ acts on the total space of the line bundle 
$\theta\otimes L_j$, ad it suffices to show that the centre $\ZZ/2$ acts trivially.

We know that $\ZZ/2$ acts trivially on $H^0 ( \theta\otimes L_j)$, hence also on the image of
$H^0 ( \theta\otimes L_j) \otimes \hol_C \ra \theta\otimes L_j$. As this image is Zariski dense, it follows
that $\ZZ/2$ acts trivially on $\theta\otimes L_j$.

STEP X) Take $j$ such that $\theta\otimes L_j$ admits a $G$-linearization. Then, if $\theta$ admits a linearization,
then also $L_j$ admits a linearization,
 contradicting Step V.

\end{proof}

\subsection{Final remark}There remains the question whether there is a $G$- invariant Thetacharacteristic $\theta$ inside $Pic^3(C)$.
We have seen that it must be $H^0(\theta) = 0$, and that $\theta \otimes L_j$ are $G$ - invariant, for $j=1,2$.

Then $H^0(\theta \otimes L_j)$ is either an irreducible $G$- representation of $G$, or it splits as an irreducible 
2-dimensional representation of $SL(2,5)$ plus a trivial 1-dimensional summand. 

Since $S^2(H^0(K)) = \CC \oplus V \oplus U$, the second alternative is not possible.

If the two representations were the same, their tensor product would be contained in $S^2(H^0(K))$ and would contain the second wedge product, which has dimension 3, a contradiction.

A computation with characters shows that the tensor product of the two distinct  irreducible 3-dimensional representations
is indeed isomorphic to  $ V \oplus U$ which is contained in $S^2(H^0(K))$, hence we do not get any contradiction in this way.

\end{document}